\documentclass[10pt,a4paper,twoside]{amsart}

\usepackage[utf8]{inputenc}
\usepackage[english]{babel}

\usepackage{amsthm}
\usepackage{amstext}
\usepackage{amsmath}
\usepackage{amsfonts}
\usepackage{amssymb}
\usepackage{mathtools}

\usepackage{tikz}
\usepackage{graphics}
\usepackage{wrapfig}
\usetikzlibrary{matrix,arrows,decorations.pathmorphing}

\usepackage[a4paper]{geometry}
\usepackage{tikz-cd}
\usepackage{epigraph}

\usepackage{hyperref}
\usepackage{cite}

\usepackage{enumerate}
\usepackage{enumitem}

\author{Matteo Tamiozzo}
\title{Special curves in modular surfaces}
\date{}

\swapnumbers
\newtheorem{thm}[subsubsection]{Theorem}     
\theoremstyle{plain}                    
\newtheorem{prop}[subsubsection]{Proposition}    
\newtheorem{corol}[subsubsection]{Corollary}     
\theoremstyle{definition}               
\newtheorem{defin}[subsubsection]{Definition}
\newtheorem{ex}[subsubsection]{Example}    
\theoremstyle{remark}                   
\newtheorem{rem}[subsubsection]{Remark}      

\newcommand{\R}{\mathbf{R}}     
\newcommand{\C}{\mathbf{C}}     
\newcommand{\Q}{\mathbf{Q}}      
\newcommand{\Z}{\mathbf{Z}}    
\newcommand{\A}{\mathbf{A}}
\newcommand{\V}{\mathcal{V}}
\newcommand{\Hp}{\mathbf{H}}
\newcommand{\Cg}{\mathcal{C}}

\newcommand{\Addr}{{\bigskip
\footnotesize
\textsc{Department of Mathematics, Imperial College London, London SW7 2AZ, UK}

\textit{Email address:} \texttt{m.tamiozzo@imperial.ac.uk}
}}

\numberwithin{equation}{subsection}

\begin{document}

\begin{abstract}
We show that geodesics in $\Hp$ attached to a maximal split torus or a real quadratic torus in $GL_{2, \Q}$ are the only irreducible algebraic curves in $\Hp$ whose image in $\R^2$ via the $j$-invariant is contained in an algebraic curve.
\end{abstract}

\maketitle

\tableofcontents

\section{Introduction}

Maximal tori in $GL_{2, \Q}$ arise from embeddings of \'etale $\Q$-algebras $E$ of degree 2 in $M_2(\Q)$. Such an algebra is either isomorphic to $\Q \times \Q$, or to an imaginary quadratic field $K$, or to a real quadratic field $F$. If $E=K$ the unique fixed point $\tau$ of $K^\times$ acting on the Poincaré upper half-plane $\Hp$ has coordinates in $\bar{\Q}$, and the classical theory of complex multiplication tells us that the $j$-invariant $j(\tau)$ also belongs to $\bar{\Q}$ (more precisely, to an abelian extension of $K$). Schneider proved that all the $\bar{\Q}$-points in $\Hp$ whose $j$-invariant is an algebraic number are obtained in this way \cite{sch37}; they are called \emph{special points}.

If $E=\Q \times \Q$ or $E=F$ one can instead attach to it a geodesic in $\Hp$, which we call a \emph{special geodesic}. It is the unique geodesic in $\Hp$ whose endpoints in $\mathbf{P}^1(\R)$ are fixed by the action of $E^\times$. We show the following properties of special geodesics.
\begin{enumerate}
\item They are the only geodesics in $\Hp$ containing infinitely many special points; dually, special points are the only ones belonging to infinitely many special geodesics.
\item They are \emph{weakly bialgebraic} curves in $\Hp$, i. e. they are irreducible algebraic curves in $\Hp\simeq \{(x, y) \in \R^2 \mid y>0\}$, and their projection to any modular curve, seen as a surface over $\R$, is contained in an algebraic curve.
\item They are the only weakly bialgebraic curves in $\Hp$.
\item An irreducible curve in $\R^2$ containing infinitely many images via the $j$-invariant of special points must contain the image of a special geodesic.
\end{enumerate}
In the last part of the document we describe the Zariski closure of the image of a special geodesic in the modular curve with full level. We give a condition implying that the image of such a geodesic is algebraic, and we examine some examples in which this property fails: for a geodesic attached to a real quadratic field, non-triviality of the class group turns out to be related with this failure.

Similar results have been studied and established in great generality for subvarieties of Shimura varieties seen as \emph{complex} algebraic varieties: we refer the reader to the survey \cite{kuy18} and the references therein. The main observation in this note is that in the simple case of modular curves, regarding the relevant spaces as real algebraic varieties instead one still obtains a group-theoretic characterisation of (weakly) bialgebraic subvarieties - although the groups which appear are not those classically considered in the theory of Shimura varieties.

This being said, the heart of our proof of the facts listed above rests on a ``base change'' argument allowing to reduce ourselves to a similar problem on the product of two (complex) modular curves, for which (an analogue of) the Ax-Lindemann-Weierstrass theorem and the André-Oort conjecture are known.

The same idea can be used to characterise irreducible algebraic curves in $\R^2$ whose image via the exponential map is algebraic; we treat this simpler case first at the beginning of the document. Finally, let us mention that geodesics in modular curves have recently been used for arithmetic purposes in \cite{rick20, dv21}.

\subsection*{Notations and conventions}

By a subvariety of an algebraic variety we always mean a \emph{closed} subvariety. We will often tacitly identify complex algebraic varieties with their complex points - but we will be careful to distinguish between real algebraic varieties and their real points. We will denote by $\Z_{>0}$ (resp. $\R_{>0}$) the set of positive integers (resp. positive real numbers).

\subsection*{Acknowledgements} We started exploring the topics discussed in this document inspired by \cite{mo20}. We are very grateful to Gregorio Baldi for bringing \cite{mo20} to our attention, and for introducing us to functional transcendence in the real world. We are indebted to Luca Dall'Ava for several stimulating discussions and for the realisation of many pictures, some of which are included in this document. We thank Leonardo Lerer for many helpful discussions and for answering several of our (naive) questions. We also wish to thank Louis Jaburi and Alice Pozzi for interesting discussions related to the contents of this note. The author's research is supported by the ERC Grant 804176.

\section{The exponential map}

\subsection{The complex situation}\label{comp-exp} Let us consider the exponential
\begin{align*}
\exp: \C & \rightarrow \C^\times\\
z & \mapsto e^{2\pi i z}
\end{align*}
and, for every $n \geq 1$, the map $\exp^n: \C^n \rightarrow (\C^\times)^n$ which equals $\exp$ on each component. We identify the source (resp. target) with the set of $\C$-points of $\A^n_\C$ (resp. $\mathbf{G}_{m, \C}^n$). As in \cite[Definition 4.3]{kuy18}, an irreducible algebraic variety $V \subset \A^n_\C$ (resp. $W \subset \mathbf{G}_{m, \C}^n$) is called \emph{bialgebraic} if $\exp^n(V)$ is algebraic (resp. each analytic irreducible component of the preimage of $W$ via $\exp^n$ is algebraic).

\begin{thm}(Ax-Lindemann-Weierstrass)\cite[Corollary 4.1.2, Theorem 4.1.3]{bt20}\label{alw-exp}
\begin{enumerate}
\item An irreducible algebraic variety $V \subset \A^n_\C$ (resp. $W \subset \mathbf{G}_{m, \C}^n$) is bialgebraic if and only if it is a translate of a $\C$-vector subspace of $\C^n$ defined over $\Q$ (resp. a translate of a subtorus of $\mathbf{G}_{m, \C}^n$).
\item Let $V \subset \A^n_\C$ be an irreducible algebraic variety. The Zariski closure of $\exp^n(V)$ is bialgebraic.
\end{enumerate}
\end{thm}

\subsection{The real situation} We are now interested in studying the situation described in \ref{comp-exp} for $n=1$, regarding the relevant spaces as \emph{real} rather than \emph{complex} varieties. Precisely, let us consider the algebraic group $\mathrm{Res}_{\C/\R}\mathbf{G}_{m, \C}=\mathrm{Spec} \;  \R[X, Y]\left[\frac{1}{X^2+Y^2}\right]$, which will be denoted by $G$ in this section. We identify $\R^2=\A^2(\R)$ with $\C$ via the map sending $(x, y)$ to $x+iy$. We consider the map
\begin{align*}
E: \R^2 & \rightarrow G(\R)=\R^2 \smallsetminus \{0\}\\
(x, y) & \mapsto (\mathrm{Re} \exp(x+iy), \mathrm{Im} \exp(x+iy)).
\end{align*}

Our aim is to determine irreducible algebraic subvarieties of $\A^2_\R$ whose image via $E$ is an algebraic subvariety of $G$.

\begin{ex}\label{expbialg-ex}
For every $t \in \R$ let $L_t=\{(x, t), x \in \R\}$. This is an irreducible algebraic subvariety of $\R^2$, whose image
\begin{equation*}
E(L_t)=\{(x, y) \in \R^2 \mid x^2+y^2=e^{-2\pi t}\}
\end{equation*}
is also algebraic.

A vertical line $S_t=\{(t, y), y \in \R\}$ has image the half-line $R_t=\{r e^{2\pi i t}, r>0\}\subset G(\R)$, which is semi-algebraic but not algebraic.
\end{ex}

\begin{defin}
A subset $\mathcal{V}\subsetneq \R^2$ is called \emph{strongly bialgebraic} if it satisfies the following conditions:
\begin{enumerate}
\item there exist algebraic subvarieties $V \subset \A^2_\R, W \subset G$ such that $\mathcal{V}=V(\R)$ and $E(\mathcal{V})=W(\R)$;
\item $\V$ cannot be written in the form $\V=V_1(\R)\cup V_2(\R)$ where $V_1, V_2 \subset \A^2_\R$ are algebraic subvarieties and the inclusions $V_i(\R) \subset \V$ are proper for $i=1, 2$.
\end{enumerate}
We call $\mathcal{V}\subsetneq \R^2$ \emph{weakly bialgebraic} if it satisfies the following conditions:
\begin{enumerate}
\item there exist algebraic subvarieties $V \subset \A^2_\R, W \subsetneq G$ such that $\mathcal{V}=V(\R)$ and $E(\mathcal{V})\subset W(\R)$.
\item $\V$ cannot be written in the form $\V=V_1(\R)\cup V_2(\R)$ where $V_1, V_2 \subset \A^2_\R$ are algebraic subvarieties and the inclusions $V_i(\R) \subset \V$ are proper for $i=1, 2$.
\end{enumerate}
\end{defin}

\begin{rem}
The notion of strongly (resp. weakly) bialgebraic subset we introduced mirrors the definition given in the complex setting in \cite[Definition 4.3]{kuy18} (resp. \cite[Definition 1.2.8]{bt20}) - as in \cite{kuy18}, we require irreducibility in our definition. As Example \ref{expbialg-ex} shows, these two notions do not coincide in our situation: every point in $\R^2$ is strongly bialgebraic, and the lines $L_t$ are strongly bialgebraic. On the other hand the lines $S_t$ are weakly bialgebraic but not strongly bialgebraic. We will now show that there are no other examples of weakly bialgebraic subsets $\mathcal{V}\subsetneq \R^2$.
\end{rem}

\begin{thm}\label{special-real-exp}
Let $\V\subset \R^2$ be a weakly bialgebraic subset which is not a singleton. Then either $\V=L_t$ for some $t \in \R$ or $\V=S_t$ for some $t \in \R$.
\end{thm}
\begin{proof}
\emph{Step 0: preliminaries.} Take $\V\subset \R^2$ weakly bialgebraic, and suppose that it is not a point. By assumption $\V$ is the set of common real zeros of finitely many non-constant polynomials in $\R[X, Y]$. The sum of their squares $S \in \R[X, Y]$ is a polynomial whose set of real zeros is $\V$. As $\V$ is irreducible by definition, it is the vanishing locus of an irreducible factor $P \in \R[X, Y]$ of $S$. Furthermore the set $\V$ is infinite, hence such a $P$ is irreducible in $\C[X, Y]$ as well. In addition there exists a smooth point $(x, y) \in \V$, and a neighbourhood of $(x, y)$ in $\V$ is diffeomorphic to the open interval $(0, 1)$. Finally, by assumption there exists a non-constant polynomial $Q \in \R[X, Y]$ such that $E(\V) \subset \{(x, y) \in \R^2 \smallsetminus \{0\}\mid Q(x, y)=0\}$.

\emph{Step 1: base change.} For $(x, y) \in \R^2$ we have $\overline{\exp(x+iy)}=\exp(-x+iy)$, hence
\begin{equation*}
\mathrm{Re} \exp(x+iy)=\frac{\exp(x+iy)+\exp(-x+iy)}{2}, \; \; \; \mathrm{Im} \exp(x+iy)=\frac{\exp(x+iy)-\exp(-x+iy)}{2i}.
\end{equation*}
We introduce the maps
\begin{align*}
f: \C^2 & \rightarrow \C^2 & g: \C^2 & \rightarrow \C^2\\
(v, w) & \mapsto (v+iw, -v+iw) & (a, b) & \mapsto \left(\frac{a+b}{2}, \frac{a-b}{2i}\right)
\end{align*}
so that the restriction of $g \circ \exp^2 \circ f$ to $\R^2\subset \C^2$ equals $E$.
Let us denote by $C_P$ (resp. $C_Q$) the complex plane curve with equation $P=0$ (resp. $Q=0$). We know that $Q \circ E : \R^2 \rightarrow \R$ vanishes on $\V$, hence the same is true for the map $h=Q \circ g \circ \exp^2 \circ f: \C^2 \rightarrow \C$. It follows that the vanishing locus of the holomorphic function $h$ restricted to the complex curve $C_P$ contains a subset diffeomorphic to $(0, 1)$. As $C_P$ is irreducible it is connected in the Euclidean topology, hence $h$ must vanish identically on $C_P$.

\emph{Step 2: application of Ax-Lindemann-Weierstrass.} The outcome of the previous step is that the image of $C_P$ via the map $g \circ \exp^2 \circ f$ is contained in $C_Q$. In other words $\exp^2(f(C_P)) \subset g^{-1}(C_Q)$. Hence the Zariski closure $Z \subset \mathbf{G}_{m, \C}^2$ of $\exp^2 \circ f(C_P)$ is contained in $g^{-1}(C_Q)$. It follows from Theorem \ref{expbialg-ex} that $Z$ is bialgebraic, and it is a translate of a one-dimensional subtorus of $\mathbf{G}_{m, \C}^2$. The curve $f(C_P)$ is connected because $C_P$ is, hence $f(C_P)$ is contained in a translate of a one-dimensional subspace of $\C^2$ defined over $\Q$.

\emph{Step 3: the final computation.} By construction $\V$ consists of the real points of $C_P$, so $f(\V)\subset f(C_P)$. A line $L \subset \A^2_\C$ with slope in $\Q \cup \{\infty\}$ is either vertical or has equation $Y=rX+\alpha$ with $r \in \Q$ and $\alpha=a+ib \in \C$. Notice that  $f(C_P)$ cannot be contained in a vertical line: indeed in this case there would be $\alpha=a+ib \in \C$ such that every point $(x, y) \in \V$ satisfies $x+iy=a+ib$, contradicting the assumption that $\V$ is not a point. Now assume that $L: Y=rX+\alpha$ is not a vertical line and $f(C_P)\subset L$; then every point $(x, y) \in \V$ satisfies
\begin{equation*}
-x+iy=r(x+iy)+(a+ib)\Rightarrow (r+1)x=-a, \; (r-1)y=-b.
\end{equation*}
As we are supposing $\V$ not to be a point, we must have either $r=-1$, hence $y=\frac{b}{2}$, or $r=1$ and $x=-\frac{a}{2}$. Therefore $\V$ is one of the lines considered in Example \ref{expbialg-ex}.
\end{proof}

\subsubsection{Special points and bialgebraicity for the exponential} By the Gelfond–Schneider theorem if $\alpha \in \bar{\Q}$ is not rational then $e^{2\pi i \alpha}$ is transcendental. In other words the only points $(x, y) \in \R^2$ with algebraic coordinates and such that $E(x, y)$ has algebraic coordinates are those with $x \in \Q$ and $y=0$. The image of such a point in $G(\R)$ is a torsion point - i. e. a root of unity - which will be called a \emph{special point}. With this terminology, Theorem \ref{special-real-exp} has the following consequence.

\begin{corol} Let $W \subsetneq G$ be an irreducible subvariety with more than one real point.
\begin{enumerate}
\item The set $W(\R)$ is the image of a strongly bialgebraic subset of $\R^2$ if and only if it is a $G(\R)$-translate of the unit circle $S^1(\R)$.
\item If $W$ satisfies the equivalent conditions in $(1)$, then $W$ contains a special point if and only if it contains infinitely many special points.
\item If $W(\R)$ contains infinitely many special points then $W(\R)=S^1(\R)$.
\end{enumerate}
\end{corol}
\begin{proof}
The first point is a direct consequence of Theorem \ref{special-real-exp}, and the second point follows immediately. Finally, $(3)$ follows from the fact that the Zariski closure of an infinite set of roots of unity in $G$ is $S^1$.
\end{proof}

\begin{rem}
The third point in the above corollary is an analogue in our situation of the Manin-Mumford conjecture for complex tori. Observe that the unique irreducible subvariety of $G$ containing infinitely many special points arises from the embedding of groups $S^1 \subset G$, hence one could regard it as a \emph{special subvariety}. The image of the inclusion $\mathbf{G}_{m}(\R)\subset G(\R)$ instead coincides with the Zariski closure of the image via $E$ of the weakly bialgebraic subvariety $S_0$; however it only contains finitely many special points.
\end{rem}

\section{The $j$-invariant}

\subsection{Setup} In this section we set $G=GL_{2, \Q}$; if $H$ is a subgroup of $G(\R)$ we will denote by $H^+\subset H$ the subgroup of matrices with positive determinant. Let
\begin{equation*}
\Hp=\{z=x+iy \in \C \mid y>0\}\simeq G(\R)/\R_{>0} O_2(\R)
\end{equation*}
be the Poincaré upper half-plane, with hyperbolic metric $\frac{\mathrm{d}x^2+\mathrm{d}y^2}{y^2}$. We identify $\Hp$ with a subset of the $\R$-points of $\A^2_\R$; this allows us in particular to talk about algebraic subvarieties of $\Hp$. For each congruence subgroup $\Gamma \subset SL_2(\Z)$ the quotient $\Gamma \backslash \Hp$ is a Riemann surface, which is the analytification of an algebraic curve over $\C$ whose Weil restriction to $\R$ will be denoted by $Y_\Gamma$. Hence $Y_\Gamma$ is a real surface with $Y_\Gamma(\R) \simeq \Gamma \backslash \Hp$; we will call it the \emph{modular surface of level} $\Gamma$. Let $p_\Gamma: \Hp \rightarrow Y_\Gamma(\R)$ be the projection map.

\begin{rem}
As a side remark, let us notice that the varieties $Y_\Gamma$ admit models over number fields. Indeed, each Riemann surface $\Gamma \backslash \Hp$ has a (canonical) model $\tilde{Y}_{\Gamma, \Q(\zeta_n)}$ over a suitable cyclotomic field $\Q(\zeta_n)$, which we may take to be different from $\Q$. Let $\Q(\zeta_n)^+\subset\Q(\zeta_n)$ be the maximal totally real subfield. Then we claim that $Y_{\Gamma, \Q(\zeta_n)^+}=\mathrm{Res}_{\Q(\zeta_n)/\Q(\zeta_n)^+}\tilde{Y}_{\Gamma, \Q(\zeta_n)}$ is a model of $Y_\Gamma$. Indeed, for any $\R$-scheme $S$, we have canonical identifications
\begin{align*}
Hom_\R(S, Y_{\Gamma, \Q(\zeta_n)^+}\times_{\Q(\zeta_n)^+}\R)&=Hom_{\Q(\zeta_n)^+}(S, Y_{\Gamma, \Q(\zeta_n)^+})\\
&=Hom_{\Q(\zeta_n)}(S\times_{\Q(\zeta_n)^+}\Q(\zeta_n), \tilde{Y}_{\Gamma, \Q(\zeta_n)})\\
&=Hom_{\Q(\zeta_n)}(S\times_{\R}\C, \tilde{Y}_{\Gamma, \Q(\zeta_n)})\\
&=Hom_{\C}(S\times_{\R}\C, \tilde{Y}_{\Gamma, \Q(\zeta_n)}\times_{\Q(\zeta_n)}\C)\\
&=Hom_{\R}(S, Y_\Gamma).
\end{align*}
\end{rem}

\begin{defin}
Let $\Gamma \subset SL_2(\Z)$ be a congruence subgroup. A subset $\mathcal{V}\subsetneq \Hp$ is called \emph{weakly bialgebraic} if it satisfies the following conditions:
\begin{enumerate}
\item there exist algebraic subvarieties $V \subset \A^2_\R, W \subsetneq Y_\Gamma$ such that $\mathcal{V}=V(\R) \cap \Hp$ and $p_{\Gamma}(\mathcal{V})\subset W(\R)$.
\item $\V$ cannot be written in the form $\V=(V_1(\R)\cap \Hp)\cup (V_2(\R)\cap \Hp)$ where $V_1, V_2 \subset \A^2_\R$ are algebraic subvarieties and the inclusions $V_i(\R)\cap \Hp \subset \V$ are proper for $i=1, 2$.
\end{enumerate}
\end{defin}

\begin{rem}\label{bialg-indep-lev}
Observe that the fact that a subset $\V \subset \Hp$ is weakly bialgebraic does not depend on $\Gamma$: assume that there exists $\Gamma \subset SL_2(\Z)$ congruence subgroup and $W \subsetneq Y_\Gamma$ such that $\mathcal{V}=V(\R) \cap \Hp$ and $p_{\Gamma}(\mathcal{V})\subset W(\R)$. Let $\Gamma' \subset SL_2(\Z)$ be another congruence subgroup. Then $\Gamma \cap \Gamma'$ has finite index in $\Gamma$, hence the maps $p_{1}:Y_{\Gamma \cap \Gamma'}\rightarrow Y_\Gamma$ and $p_2:Y_{\Gamma \cap \Gamma'}\rightarrow Y_{\Gamma'}$ are finite. We have
\begin{equation*}
p_{\Gamma'}(\V)=p_{2}\circ p_{\Gamma \cap \Gamma'}(\V) \subset (p_2\circ p_1^{-1}(W))(\R)
\end{equation*}
and $p_2\circ p_1^{-1}(W)\subsetneq Y_{\Gamma'}$ is algebraic.
\end{rem}
Our aim is to describe weakly bialgebraic subsets of $\Hp$. In view of the above remark it suffices to consider the case $\Gamma=SL_{2}(\Z)$, so that $Y_{\Gamma}=\A^2_{\R}$ and $p_\Gamma: \Hp \rightarrow \R^2$ is the map sending $z$ to $(\mathrm{Re} \; j(z), \mathrm{Im} \; j(z))$. We will focus on this case in what follows, with the exception of \ref{bern-lemn} below. We start by describing two concrete examples illustrating the general phenomena we will later study.

\subsection{Two examples}

\subsubsection{The positive imaginary axis}\label{posax} The $j$-invariant is injective and takes real values on the half-line $\{(x, y) \in \R^2, x=0, y \geq 1\}$, and $j(it)$ goes to infinity as $t>0$ goes to infinity. Furthermore $j(i)=1728$ and $j(-1/z)=j(z)$ for $z \in \Hp$. It follows that the image via $j$ of the vertical half-line $\mathcal{C}_{0, \infty}=\{(x, y) \in \R^2 \mid x=0, y>0\} \subset \Hp$ is the semi-algebraic set $\{(x, y) \in \R^2 \mid y=0, x \geq 1728\}$. In particular $\Cg_{0, \infty}$ is a weakly bialgebraic curve in $\Hp$. Notice that $j(\Cg_{0, \infty})$ is not algebraic; the smallest algebraic set containing it is the axis $Y=0$ in $\R^2$. This is the image of the union $\Cg_{0, \infty} \cup \partial^+$, where $\partial^+$ is the part of the boundary of the usual fundamental domain for the $SL_2(\Z)$-action with positive first coordinate. Let us end this example by pointing out that $\Cg_{0, \infty}$ is the unique geodesics in $\Hp$ with endpoints $0, \infty$, and it contains infinitely many special points. It can also be described as the orbit of any of its points via the action of the group $\left\{\begin{pmatrix}
t & 0\\
0 & t^{-1}
\end{pmatrix}, t \in \R^\times \right\}$ of matrices of determinant one in $T(\R)$, where $T \subset G$ is the diagonal torus.

\subsubsection{Bernoulli's lemniscate}\label{bern-lemn} We learned the example we are going to explain now from the Mathoverflow post \cite{mo20}, which was the starting point of our investigations.\footnote{We thank Gregorio Baldi for pointing it out to us.} Let us take $\Gamma=\Gamma(2)$, so that $Y_{\Gamma(2)}(\R)=\C \smallsetminus \{0, 1\}$ and the projection map is Klein's $\lambda$-function. Let us consider the matrix $A=\begin{pmatrix}
1 & 1\\
1 & -1
\end{pmatrix}$, and the curve $\Cg_A \subset \Hp$ with equation $Az=\bar{z}$; it is the half-circle with centre $(1, 0)$ and radius $\sqrt{2}$. In other words it is the only geodesic in $\Hp$ with endpoints $1\pm \sqrt{2} \in \R$. It can be described in a slightly different way, which will be generalised later on: we have an embedding
\begin{align*}
\iota: F=\Q(\sqrt{2}) & \rightarrow M_2(\Q)\\
a+b\sqrt{2} & \mapsto a \mathrm{Id}+bA
\end{align*}
inducing a map $\iota_\R: F^1_\R=\{x \in F \otimes_\Q \R \mid N(x)=1\} \rightarrow G(\R)^+$. The geodesic $\Cg_A$ is the $F^1_\R$-orbit of any of its points. In particular $3\mathrm{Id}+2A=\begin{pmatrix}
5 & 2\\
2 & 1
\end{pmatrix} \in \Gamma(2) \cap F^1_\R$ fixes $\Cg_A$. It follows that the image of $\Cg_A$ in $Y_{\Gamma(2)}$ coincides with the image of the compact set $F^1_\R/(\pm \varepsilon^\Z)$, where $\varepsilon=3+2\sqrt{2}$, therefore it is a compact geodesic in $Y_{\Gamma(2)}$. In fact $p_{\Gamma(2)}(\Cg_A)$ is the lemminscate with equation $(X^2+Y^2)((X-1)^2+Y^2)=\frac{1}{16}$. In particular $\Cg_A$ is bialgebraic. Finally, notice that, as in the previous example, the curve $\Cg_A$ contains infinitely many special points.

\subsection{Special geodesics} Geodesics in $\Hp$ are either vertical lines or half-circles with centre on the real axis. They can be described as follows: given $A \in GL_2(\R)$ with trace zero and hyperbolic (i. e. with negative determinant) let us consider the curve $\mathcal{C}_A\subset \Hp$ with equation $Az=\bar{z}$. Writing $z=x+iy$ and $A=\begin{pmatrix}
a & b\\
c & -a
\end{pmatrix}$ we see that points $x+iy \in \mathcal{C}_A$ satisfy the equation
\begin{equation*}
\frac{az+b}{cz-a}=\bar{z} \Leftrightarrow c|z|^2-2 a \mathrm{Re} z-b=0\Leftrightarrow c(x^2+y^2)-2ax-b=0.
\end{equation*}
Hence $\mathcal{C}_A$ is the unique geodesic with endpoints the fixed points of $A$ in $\mathbf{P}^1(\R)$. Every geodesic in $\Hp$ can be obtained in this way: indeed a non-vertical geodesic has equation of the form $(X-x_0)^2+Y^2=r$, and it suffices to choose $a, b, c$ such that $\frac{a}{c}=x_0$ and $\frac{bc+a^2}{c^2}=r$. A vertical geodesic has equation $X=x_0$, so we may choose $a=1, b=-2x_0, c=0$.

\begin{rem}
An alternative way to think about geodesics in $\Hp$ is the following: for every $A \in GL_2(\R)$ with trace zero and negative determinant the map $z \mapsto A\bar{z}$ is an anti-holomorphic involution of $\Hp$, and the curve $\Cg_A$ is the set of fixed points of this involution. This is the point of view adopted by Jaffee \cite{jaf74}, who first suggested to use this idea to study the arithmetic properties of symmetric spaces attached to certain reductive groups when they do not carry a complex structure.
\end{rem}

\subsubsection{Geodesics attached to real quadratic fields} Given a real quadratic field $F=\Q(\sqrt{d})$ and an embedding $\iota: F \rightarrow M_2(\Q)$ the element $\iota(\sqrt{d}) \in GL_2(\R)$ has trace zero and is hyperbolic. The geodesic attached to $\iota(\sqrt{d})$ via the previous construction will be denoted by $\Cg_{F, \iota}$. This construction has the following group-theoretic interpretation (which in particular shows that $\Cg_{F, \iota}$ only depends on $F$ and $\iota$, and not on the choice of $d$). Consider the torus $T_F=\mathrm{Res}_{F/\Q}\mathbf{G}_{m, F}$; the embedding $\iota$ induces an embedding of groups, abusively denoted by the same symbol, $\iota: T_F \hookrightarrow G$, which in turn induces an map
\begin{equation*}
T_F(\R)/\R^\times \rightarrow G(\R)/\R^\times.
\end{equation*}
There are two points in $\R\subset \mathbf{P}^1(\R)$ fixed by $T_F(\Q)$, and $\Cg_{F, \iota}$ is the unique geodesic in $\Hp$ having these endpoints. The image of the subgroup $T_F(\R)^+/\R^\times$ via the above map acts on $\Hp$; the geodesic $\Cg_{F, \iota}$ is the $T_F(\R)^+/\R^\times$-orbit of any of its points.
\begin{rem}
\begin{enumerate}
\item In particular, in the above construction of $\Cg_{F, \iota}$ we may choose $d$ such that $A=\iota(\sqrt{d})$ belongs to $M_2(\Z)$ and the greatest common divisor of the coefficients of $A$ is one.
\item All the embeddings of $F$ in $M_2(\Q)$ are conjugate, hence the resulting geodesics are translates of a given one via the action of $G(\Q)^+$.
\end{enumerate}
\end{rem}

\subsubsection{Geodesics attached to split tori} Another distinguished class of geodesics consists of vertical geodesics with rational first coordinate and half-circles with rational endpoints on the real axis (which give rise to modular symbols). In group-theoretic terms, those arise from embeddings of the split torus $T=\mathbf{G}_{m, \Q}\times \mathbf{G}_{m, \Q}$ in $G$, as follows: letting $\iota: T \hookrightarrow G$ be an embedding, there is a unique geodesic $\mathcal{C}_{\Q \times \Q, \iota}$ in $\Hp$ with endpoints the points of $\mathbf{P}^1(\R)$ fixed by $\iota(T(\Q))$. For example, if $\delta: T \hookrightarrow G$ is the diagonal embedding, then the group $\delta(T(\Q))$ fixes the points $0, \infty \in \mathbf{P}^1(\R)$, and the geodesic $\mathcal{C}_{\Q \times \Q, \delta}$ is the vertical line with equation $X=0$. As above, this is also the $T(\R)^+/\R^\times$-orbit of any of its points. Any other embedding $\iota: T \rightarrow G$ is $G(\Q)^+$-conjugate to $\delta$, hence the attached geodesic is a $G(\Q)^+$-translate of $\mathcal{C}_{\Q \times \Q, \delta}$. This construction accounts for all the geodesics in $\Hp$ whose endpoints in $\mathbf{P}^1(\R)$ are rational.

\begin{defin}
A geodesic in $\Hp$ obtained from an embedding of $T$ or $T_F$ for some real quadratic field $F$ via the above construction is called a \emph{special geodesic}.
\end{defin}

\begin{rem}\label{eqn-spec-geod}
Concretely, special geodesics in $\Hp$ have equation $Az=\bar{z}$ for some hyperbolic trace-zero matrix $A \in M_2(\Q)$ (which can be rescaled so as to have integral coprime coefficients). Such a geodesic is attached to a split torus (resp. to a torus coming from a real quadratic field) if the absolute value of the determinant of $A$ is (resp. is not) a square in $\Q$. Notice that special points satisfy instead equations of the form $Az=z$ with $A \in M_2(\Q)$ elliptic trace-zero matrix.
\end{rem}

Observe that every geodesic in $\Hp$ is an algebraic curve. The following proposition gives a first characterisation of special geodesics among all geodesics in $\Hp$.

\begin{prop}\label{spec-geod-char}
Let $\mathcal{C}\subset \Hp$ be a geodesic. The following are equivalent:
\begin{enumerate}
\item $\Cg$ is special.
\item $\Cg$ contains infinitely many special points.
\item $\Cg$ is defined over $\Q$.
\item The endpoints of $\Cg$ in $\mathbf{P}^1(\R)$ are rational or quadratic conjugates.
\end{enumerate}
\end{prop}
\begin{proof}
By construction every special geodesic is defined over $\Q$; furthermore the endpoints of a geodesic defined by an equation with rational coefficients are either rational or quadratic conjugates. On the other hand for any two rational or quadratic conjugate points of $\mathbf{P}^1(\R)$ the unique geodesic having them as endpoints is defined over $\Q$. Now a geodesic $\Cg$ defined over $\Q$ is either vertical with rational first coordinate - hence special - or a half-circle with equation $(X-x_0)^2+Y^2=r$, with $x_0, r \in \Q$; in the latter case we may choose $a, b, c \in \Q$ such that $\frac{a}{c}=x_0$ and $\frac{bc+a^2}{c^2}=r$, and set $A=\begin{pmatrix}
a & b\\
c & -a
\end{pmatrix}$. Then $\Cg$ has equation $Az=\bar{z}$, and by Remark \ref{eqn-spec-geod} it is special. Hence we have established the equivalence of $(1), (3), (4)$.

Finally, let us show that geodesics defined over $\Q$ are precisely those containing infinitely many special points. This is clear for vertical geodesics. Now take a geodesic $\Cg$ with equation $(X-x_0)^2+Y^2=r$. If $x_0, r$ are rational then any point in $\Cg$ with rational first coordinate is special. Conversely, assume that $\Cg$ contains infinitely many special points. Then there exist $(x_1, y_1), (x_2, y_2) \in \Q^2$ with $x_1 \neq x_2$ and $y_1, y_2>0$, and $d_1, d_2 \in \Z_{>0}$ such that
\begin{equation*}
(x_i-x_0)^2+d_iy_i^2=r, \; \; i=1, 2.
\end{equation*}
Taking the difference we obtain
\begin{equation*}
x_1^2-x_2^2+d_1y_1^2-d_2y_2^2-2x_0(x_1-x_2)=0 \Rightarrow x_0 \in \Q \Rightarrow r \in \Q
\end{equation*}
so $\Cg$ is defined over $\Q$.
\end{proof}

\begin{rem}
The computations in the above proof give the following more precise information.
\begin{enumerate}
\item A geodesic is special if and only if it contains at least two distinct special points - in which case it contains infinitely many. However, unlike in the classical case \cite[Theorem 3.5]{kuy18}, containing one special point is not enough for a geodesic to be special - indeed, any half-circle with real irrational centre passing through a special point is not defined over $\Q$.
\item A special point attached to an imaginary quadratic field $K$ is contained in a special geodesic attached to a real quadratic field $F=\Q(\sqrt{d})$ if and only if $d$ is the norm of an element in $K$. 
\end{enumerate}
Let us also notice that the following ``dual" version of the above proposition holds true.
\end{rem}

\begin{prop}
Let $z \in \Hp$ be a point. The following assertions are equivalent:
\begin{enumerate}
\item $z$ is a special point.
\item $z$ belongs to infinitely many special geodesics.
\end{enumerate}
\end{prop}
\begin{proof}
Take $z=x+iy \in \Hp$ special, so that $x \in \Q$ and $y \in \R_{>0}$ has rational square; let $D=y^2 \in \Q$. For every $x_0 \in \Q$ we have $(x-x_0)^2+y^2=(x-x_0)^2+D$; letting $r=(x-x_0)^2+D$ we see that $z$ belongs to the special geodesic with equation $(X-x_0)^2+Y^2=r$.

Conversely, assume that $z=x+iy \in \Hp$ belongs to infinitely many special geodesics. Then there are two couples of rational numbers $(x_0, r), (x_0', r')$ such that $x_0 \neq x_0'$ and
\begin{equation*}
(x-x_0)^2+y^2=r, \; \; (x-x_0')^2+y^2=r'.
\end{equation*}
Taking the difference we deduce that $x \in \Q$, hence $y^2 \in \Q$ and $z$ is a special point.
\end{proof}

\subsection{The $j$-invariant and bialgebraicity}

We now wish to describe weakly bialgebraic subsets of $\Hp$. In view of Remark \ref{bialg-indep-lev}, it suffices to study weakly bialgebraic subsets for the map
\begin{align*}
J: \Hp & \rightarrow \R^2\\
z & \mapsto (\mathrm{Re} \; j(z), \mathrm{Im} \; j(z)).
\end{align*}

In order to do this we will follow the same general strategy used in the proof of Theorem \ref{special-real-exp}; in particular we will make use of of the Ax-Lindemann-Weierstrass theorem for the map
\begin{equation*}
j^2=j \times j: \Hp^2 \rightarrow \C^2,
\end{equation*}
due to Pila.

\begin{thm}(\hspace{1sp}\cite{pil11})\label{alw-modular}
Let $V \subset \A^2_\C$ be an irreducible algebraic variety. If there is a point $P \in V\cap \Hp^2$ such that $j^2$ maps an open neighbourhood of $P$ in $V$ to an algebraic subvariety properly contained in $\C^2$ then one of the following assertions holds true:
\begin{enumerate}
\item $V$ is a point.
\item $V \cap \Hp^2=\{(x, y) \in \Hp^2 \mid y=c\}$ or $V\cap \Hp^2=\{(x, y) \in \Hp^2 \mid x=c\}$ for some $c \in \C$.
\item $V \cap \Hp^2 =\{(x, y) \in \Hp^2 \mid y=Ax\}$ for some $A \in G(\Q)^+$.
\end{enumerate}
\end{thm}

\begin{thm}\label{bialg-mod-surf}
Weakly bialgebraic subvarieties of $\Hp$ are either points or special geodesics.
\end{thm}

\begin{proof}
\emph{Step 0: preliminaries.} Take $\V\subset \Hp$ weakly bialgebraic; then $\V$ is the set of zeros in $\Hp$ of an irreducible polynomial $P(X, Y) \in \R[X, Y]$. Assume that $\V$ is not a point; then it is infinite, hence $P$ is irreducible in $\C[X, Y]$. There exists a smooth point $x \in \V$, and a neighbourhood of $x$ in $\V$ is diffeomorphic to the open interval $(0, 1)$. Finally, by assumption there exists a non-constant polynomial $Q \in \R[X, Y]$ such that $J(\V) \subset \{(x, y) \in \R^2 \mid Q(x, y)=0\}$.

\emph{Step 1: base change.} Recall that for $z \in \C$ we have $\overline{j(z)}=j(-\bar{z})$. Hence for $(x, y) \in \R\times \R_{>0}$ we have
\begin{equation*}
\mathrm{Re} \; j(x+iy)=\frac{j(x+iy)+j(-x+iy)}{2}, \; \; \; \mathrm{Im} \; j(x+iy)=\frac{j(x+iy)-j(-x+iy)}{2i}.
\end{equation*}
Consider the maps
\begin{align*}
f: \C^2 & \rightarrow \C^2 & g: \C^2 & \rightarrow \C^2\\
(v, w) & \mapsto (v+iw, -v+iw) & (a, b) & \mapsto \left(\frac{a+b}{2}, \frac{a-b}{2i}\right);
\end{align*}
let $U\subset \C^2$ be the set of points $(v, w)$ such that the imaginary parts of $v+iw, -v+iw$ are positive. In particular $\R\times \R_{>0}$ is contained in $U$, and the restriction of $g \circ j^2 \circ f$ to $\R\times \R_{>0}$ equals $J$.
Let $C_P$ (resp. $C_Q$) be the complex plane curve with equation $P=0$ (resp. $Q=0$). The map $h=Q \circ g \circ j^2 \circ f: U \rightarrow \C$ vanishes on $\V$; since $\V$ contains a subset $I$ diffeomorphic to $(0, 1)$, the holomorphic map $h$ must vanish identically on the connected component $C$ of $C_P\cap U$ containing $I$.

\emph{Step 2: application of Ax-Lindemann-Weierstrass.} The outcome of the previous step is that the image of $C$ via the map $g \circ j^2 \circ f$ is contained in $C_Q$. In other words $j^2(f(C))$ is contained in the algebraic subvariety $g^{-1}(C_Q)\subsetneq \A^2_\C$. Therefore we deduce that $f(C_P)$ is a subvariety of $\A^2_\C$ as in Theorem \ref{alw-modular}.

\emph{Step 3: the final computation.} As $f(\V)\subset f(C_P)$ and $\V$ is not a point we deduce that $f(C_P)\cap \Hp^2$ has equation either $X=c$, or $Y=c$, or $Y=AX$ for some $A \in G(\Q)^+$ and $c  \in \C$. In the first case we find that every $(x, y) \in \V$ satisfies $x+iy=c$, contradicting the assumption that $\V$ is not a point; similarly, the second case cannot occur. Therefore we deduce that there exists $A \in G(\Q)^+$ such that every $(x, y) \in \V$ safisfies
\begin{equation*}
-x+iy=A(x+iy).
\end{equation*}
Let $\tilde{A}=\begin{pmatrix} -1 & 0\\
0 & 1
\end{pmatrix}A$; writing $\tilde{A}=\begin{pmatrix}
a & b\\
c & d
\end{pmatrix}$ the previous equation becomes
\begin{equation*}
c(x^2+y^2)+(d-a)x-(a+d)iy-b=0.
\end{equation*}
In particular we must have $a=-d$. By Remark \ref{eqn-spec-geod}, we deduce that every weakly bialgebraic set $\V$ which is not a point must be a special geodesic.

Finally, special geodesics are indeed weakly bialgebraic, as explained below in section \ref{mod-pol}.
\end{proof}

\begin{rem}
Let $\mathcal{V}\subset \Hp$ be a weakly bialgebraic subset which is not a singleton. The above theorem tells us in particular that $\mathcal{V}$ is an orbit for the action of a subgroup of $PSL_2(\R)$ arising from the group of real points of a split or real quadratic torus in $G$. Letting $\iota: \Hp \rightarrow \Hp \times \Hp$ be the map sending $z$ to $(z, - \bar{z})$, Step 3 of our argument rests on the observation that the pullback via $\iota$ of a subvariety with equation $Y=AX$ for $A \in G(\Q)^+$ turns out to be either a singleton or a geodesic attached to a torus as above. There is a more conceptual way to show that each connected component $\mathcal{V}^0 \subset \mathcal{V}$ must be the orbit of a connected Lie subgroup of $PSL_2(\R)$, which may be of use to generalise Theorem \ref{bialg-mod-surf}. With the same notations as in the proof, consider the commutative diagram
\begin{center}
\begin{tikzcd}
\Hp \arrow[r, "\iota"] \arrow[d, "J"]
& \Hp \times \Hp \arrow[d, "g \circ j^2"] \\
\R^2 \arrow[r] & \C^2
\end{tikzcd}
\end{center}
where the bottom arrow is the natural inclusion. Steps 1 and 2 in the proof of Theorem \ref{bialg-mod-surf} (which rely on Theorem \ref{alw-modular}) imply that $f(C_P)\cap \Hp^2$ is a totally geodesic submanifold of the symmetric space $\Hp^2$. Therefore the same is true for $\iota^{-1}(f(C_P)\cap \Hp^2)=\mathcal{V}$, so $\mathcal{V}^0 \subset \Hp$ is a connected totally geodesic submanifold. It follows from the proof of \cite[Chapter IV, Theorem 7.2]{helg78} that $\mathcal{V}^0$ is the orbit of a connected Lie subgroup of $PSL_2(\R)$ - whose Lie algebra arises from a Lie triple system as in \emph{loc. cit.}
\end{rem}

\subsection{Modular polynomials and special geodesics}\label{mod-pol}

\subsubsection{Modular curves in the modular surface} For $N \geq 1$ let $\Phi_N(X, Y) \in \Z[X, Y]$ be the $N$-th modular polynomial. Recall that $\Phi_N(X, Y)=\Phi_N(Y, X)$ if $N>1$, and $\Phi_N$ is irreducible in $\C(X)[Y]$. It follows that $\Phi_N$ is also irreducible in $\C[X, Y]$. Let $\tilde{\Phi}_N(X, Y)=\Phi_N(X+iY, X-iY)$. If $N=1$ then $\tilde{\Phi}_N(X, Y)=2iY$, whereas for $N>1$ we have
\begin{enumerate}
\item $\tilde{\Phi}_N(X, Y) \in \Z[X, Y]$ and $\tilde{\Phi}_N(X, -Y)=\tilde{\Phi}_N(X, Y)$;
\item $\tilde{\Phi}_N(X, Y)$ is irreducible in $\C[X, Y]$.
\end{enumerate}
The first (resp. second) property follows from symmetry (resp. irreducibility) of $\Phi_N$. Let
\begin{equation*}
\mathcal{Z}_N = \{\tilde{\Phi}_N=0\} \subset \mathbf{A}^2_{\R};
\end{equation*}
the above properties imply that $\mathcal{Z}_N$ is a geometrically irreducible plane curve, symmetric with respect to the axis $Y=0$.

Now take $A \in G(\Q)$ with trace zero and negative determinant and consider the curve $\Cg_A$ in $\Hp$ with equation $Az=\bar{z}$. The curve $\Cg_A$ is unchanged if we replace $A$ by $\lambda A$ for some $\lambda \in \Z_{>0}$, hence we may (and will) assume that $A$ has integral coprime coefficients. Let $N$ be the absolute value of the determinant of $A$; if $z$ is a point of $\Cg_A$ the elliptic curves attached to the lattices $\Z\oplus z\Z$ and $\Z \oplus \bar{z}\Z$ are related by an isogeny with cyclic kernel of cardinality $N$, hence we have $\Phi_{N}(j(z), j(-\bar{z}))=0$. As $j(-\bar{z})=\overline{j(z)}$ we see that $J(\Cg_A) \subset \R^2$ is contained in the curve $\mathcal{Z}_N$. This shows that every special geodesic is weakly bialgebraic. Notice however that the curve $\mathcal{Z}_N$ may contain distinct images of geodesics $\Cg_A$; this is related to the existence of trace-zero matrices of determinant $-N$ which are not $SL_2(\Z)$-conjugate. We examine this phenomenon below.

\subsubsection{The curve $\mathcal{Z}_1$} For $N=1$ we obtain the curve $Y=0$. If $A \in M_2(\Z)$ has trace zero and determinant $-1$ then $J(\Cg_A)$ is contained in $\mathcal{Z}_1$. If we take $A^1_1=\begin{pmatrix}
-1& 0\\
0 & 1
\end{pmatrix}$ then the curve $\Cg_{A^1_1}$ is the positive imaginary axis. As observed in Example \ref{posax}, its image is the closed half-line $\{(x, y) \in \R^2 \mid y=0, x \geq 1728\}$. For $A^1_2=\begin{pmatrix}
0 & 1\\
1 & 0
\end{pmatrix}$ the curve $\Cg_{A^1_2}$ is the intersection of the unit circle with centre at the origin and the upper half-plane. Notice that $A^1_1$ and $A^1_2$ are not $SL_2(\Z)$-conjugate, and the images $J(\Cg_{A^1_1})$ and $J(\Cg_{A^1_2})$ do not coincide. One checks that every matrix $A \in M_2(\Z)$ with trace zero and determinant -1 is $SL_2(\Z)$-conjugate either to $A^1_1$ or to $A^1_2$, and we have
\begin{equation*}
\mathcal{Z}_1(\R)=J(\Cg_{A^1_1}) \cup J(\Cg_{A^1_2}).
\end{equation*}
\subsubsection{The curve $\mathcal{Z}_N$ for $N$ squarefree} Take $N>1$ squarefree; the ring $\Z[\sqrt{N}]$ is an order in the real quadratic field $\Q(\sqrt{N})$, and every matrix in $M_2(\Z)$ with determinant $-N$ has coprime coefficients. Hence for every $SL_2(\Z)$-conjugacy class of matrices with integral coefficients, trace zero and determinant $-N$ we obtain a curve $J(\Cg_A)\subset \mathcal{Z}_N(\R)$, where $A$ is any matrix in the given conjugacy class. The set of $SL_2(\Z)$-conjugacy classes of matrices in $M_2(\Z)$ with trace zero and determinant $-N$ is in bijection with the set $Cl^+(\Z[\sqrt{N}])$ of (not necessarily invertible) narrow ideal classes of $\Z[\sqrt{N}]$ (see \cite[Remark 9]{con}). Identifying these two sets, we have an inclusion
\begin{equation*}
\bigcup_{A \in Cl^+(\Z[\sqrt{N}])} J(\Cg_A)\subset \mathcal{Z}_N(\R).
\end{equation*}
In fact, the above inclusion is an equality. Indeed, if $(x, y) \in \R^2$ satisfies $\Phi_N(x+iy, x-iy)=0$, writing $x+iy=j(z)$ we see that the elliptic curves attached to the lattices $\Z\oplus z\Z$ and $\Z \oplus (-\bar{z})\Z$ are related by an isogeny with cyclic kernel of cardinality $N$, hence we must have $-\bar{z}=\frac{az+b}{cz+d}$ for some matrix $\tilde{A}=\begin{pmatrix}
a & b\\
c & d
\end{pmatrix} \in M_2(\Z)$ with determinant $N$. The matrix $A=\begin{pmatrix} -1 & 0\\
0 & 1
\end{pmatrix}\tilde{A}$ has determinant $-N$, and $Az=\bar{z}$. Furthermore, as $z \in \Hp$ the trace of $A$ must be zero, and $(x, y)$ belongs to $J(\Cg_A)$.

Let us also observe that if $A \in M_2(\Z)$ has trace zero and determinant $-N$, then the same is true for $-A$, and $\Cg_{-A}=\Cg_A$. Letting $\widetilde{Cl}^+(\Z[\sqrt{N}])$ be the quotient of $Cl^+(\Z[\sqrt{N}])$ by the equivalence relation identifying $A$ with $-A$, we obtain the following proposition.
\begin{prop}\label{Znsqfree}
Let $N>1$ be a squarefree integer, and $\mathcal{Z}_N \subset \mathbf{A}^2_\R$ the curve with equation $\tilde{\Phi}_N(X, Y)=0$. Then
\begin{equation*}
\bigcup_{A \in \widetilde{Cl}^+(\Z[\sqrt{N}])} J(\Cg_A) = \mathcal{Z}_N(\R).
\end{equation*}
Furthermore if $A, B \in M_2(\Z)$ are two matrices with trace zero, determinant $-N$ and distinct image in $\widetilde{Cl}^+(\Z[\sqrt{N}])$ then $J(\Cg_A) \neq J(\Cg_B)$.
\end{prop}
\begin{proof}
The first assertion follows from the discussion before the statement of the proposition. To prove the last assertion, notice that a geodesic $Az=\bar{z}$ in $\Hp$ is uniquely determined by its endpoints in $\mathbf{P}^1(\R)$, which are the fixed points of $A$. It follows that two matrices with integral coefficients, trace zero and determinant $-N$ giving rise to the same geodesic in $\Hp$ must be either equal or opposite. Take $A, B \in M_2(\Z)$ having distinct image in $\widetilde{Cl}^+(\Z[\sqrt{N}])$. Then for every $M \in SL_2(\Z)$ the geodesics $M\cdot \Cg_A$ and $\Cg_B$ are distinct, hence they have at most one point of intersection. Therefore $J(\Cg_A)\neq J(\Cg_B)$.
\end{proof}

\subsubsection{The curves $\mathcal{Z}_2$ and $\mathcal{Z}_3$} As the narrow class group of $\Z[\sqrt{2}]$ is trivial, the image of any geodesic attached to a matrix with integral coefficients, trace zero and determinant -2 is the curve $\mathcal{Z}_2$.

The narrow class group of $\Z[\sqrt{3}]$ has two elements. Two matrices with trace zero and determinant $-3$ which are not $SL_2(\Z)$-conjugate are $A^3_1=\begin{pmatrix}
0 & 3\\
1 & 0
\end{pmatrix}$ and $A^3_2=\begin{pmatrix}
0 & 1\\
3 & 0
\end{pmatrix}$. However $A^3_1$ is $SL_2(\Z)$-conjugate to $-A^3_2$, which gives rise to the same geodesic as $A^3_2$. Hence
\begin{equation*}
J(\Cg_{A^3_1})=J(\Cg_{A^3_2})=\mathcal{Z}_3(\R).
\end{equation*}

Notice however that there are two $SL_2(\Z)$-orbits of \emph{oriented} geodesics $\Cg_A$ with $A$ of trace zero and determinant $-3$.

The curve $\mathcal{Z}_2$ (resp. $\mathcal{Z}_3$) is represented below on the left (resp. right).

\begin{center}
\begin{figure}[h!]
\includegraphics[width=6.2 cm]{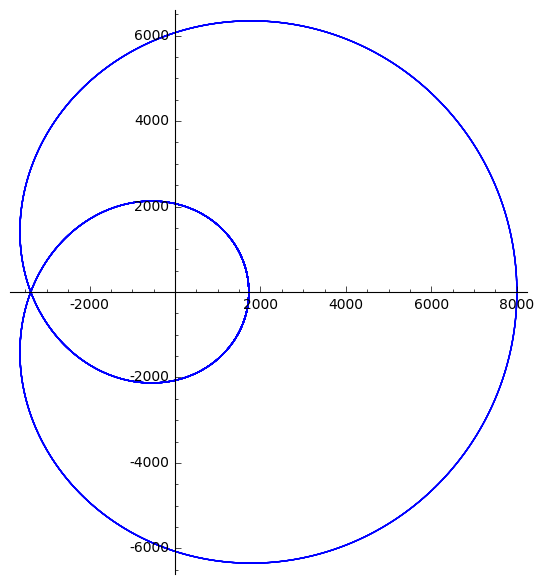} \includegraphics[width=6.2 cm]{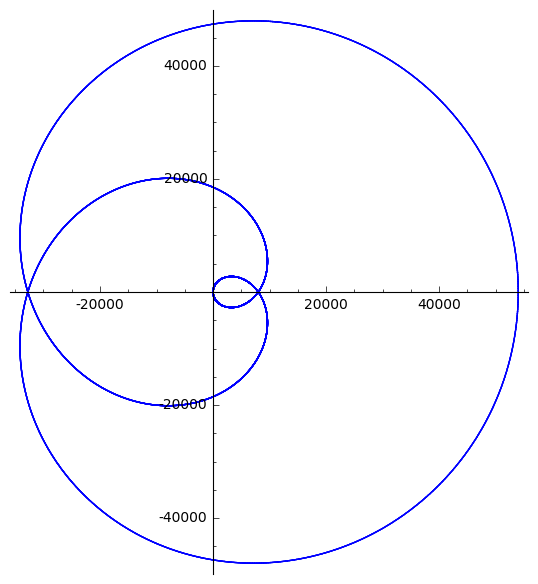}
\end{figure}
\end{center}

Generalising the previous example we obtain the following result.

\begin{corol}
Let $N>1$ be a squarefree integer congruent to 2 or 3 modulo 4. Assume that the class group of $\Z[\sqrt{N}]$ is trivial. Then $J(\Cg_A) = \mathcal{Z}_N(\R)$ for any $A \in M_2(\Z)$ with trace zero and determinant $-N$.
\end{corol}

\begin{proof}
If $\Z[\sqrt{N}]^\times$ contains an element of norm -1 then the narrow class group of $\Z[\sqrt{N}]$ is trivial, hence the statement follows from Proposition \ref{Znsqfree}. Now assume that the equation $X^2-NY^2=-1$ has no integral solution; in this case $Cl^+(\Z[\sqrt{N}])$ has two elements. Consider the two matrices $A^N_1=\begin{pmatrix}
0 & N\\
1 & 0
\end{pmatrix}$ and $A^N_2=\begin{pmatrix}
0 & 1\\
N & 0
\end{pmatrix}$. Given $\begin{pmatrix}
a & b\\
c & d
\end{pmatrix} \in SL_2(\Z)$ we have
\begin{equation*}
\begin{pmatrix}
a & b\\
c & d
\end{pmatrix}\begin{pmatrix}
0 & 1\\
N & 0
\end{pmatrix}\begin{pmatrix}
d & -b\\
-c & a
\end{pmatrix}=\begin{pmatrix}
Nbd-ac & a^2-Nb^2\\
Nd^2-c^2 & -Nbd+ac
\end{pmatrix}.
\end{equation*}
It follows that the matrices $A^N_1$ and $A^N_2$ are not $SL_2(\Z)$-conjugate, so they correspond to the two elements of $Cl^+(\Z[\sqrt{N}])$; on the other hand
\begin{equation*}
\begin{pmatrix}
0 & 1\\
-1 & 0
\end{pmatrix}\begin{pmatrix}
0 & 1\\
N & 0
\end{pmatrix}\begin{pmatrix}
0 & -1\\
1 & 0
\end{pmatrix}=\begin{pmatrix}
0 & -N\\
-1 & 0
\end{pmatrix}.
\end{equation*}
As $\Cg_{A^N_1}=\Cg_{-A^N_1}$ the result follows.
\end{proof}

\begin{rem}
For a general $N>1$ squarefree and congruent to 2, 3 modulo 4, the set $\widetilde{Cl}^+(\Z[\sqrt{N}])$ may have different cardinality from the class group of $\Z[\sqrt{N}]$. For example, assume that $\Z[\sqrt{N}]$ contains a unit of norm -1, so that $Cl^+(\Z[\sqrt{N}])=Cl(\Z[\sqrt{N}])$. Any matrix $A \in M_2(\Z)$ with trace zero and determinant $-N$ is $SL_2(\Z)$-conjugate to $-A^t$, via the matrix $\begin{pmatrix}
0 & 1\\
-1 & 0
\end{pmatrix}$. Hence identifying the equivalence class of each $A$ in $Cl^+(\Z[\sqrt{N}])$ with that of $-A$ is the same as identifying the equaivalence class of $A$ with that of $A^t$. As explained in \cite[Example 15]{con}, this amounts to identifying every element in the group $Cl(\Z[\sqrt{N}])$ with its inverse. Therefore $\widetilde{Cl}^+(\Z[\sqrt{N}])$ has smaller cardinality than $Cl^+(\Z[\sqrt{N}])=Cl(\Z[\sqrt{N}])$ as soon as the latter has an element of order greater than two. For example, take $N=82$. The matrix $\begin{pmatrix}
-1 & 27\\
3 & 1
\end{pmatrix}$ corresponds to the ideal $(3, 1+\sqrt{82})\subset \Z[\sqrt{82}]$, which has order 4 in the class group, and is not $SL_2(\Z)$-conjugate to its opposite.
\end{rem}

\subsubsection{The curve $\mathcal{Z}_{5}$} We have seen above that the curve $\mathcal{Z}_1$ is the union of two images of geodesics, each of which is not algebraic. Let us give further examples of this phenomenon.

There are two elements in $\widetilde{Cl}^+(\Z[\sqrt{5}])$ (since there is one invertible and one non-invertible ideal class for the ring $\Z[\sqrt{5}]$) corresponding to the matrices $A^5_1=\begin{pmatrix}
0 & 5\\
1 & 0
\end{pmatrix}$ and $A^5_2=\begin{pmatrix}
1 & 2\\
2 & -1
\end{pmatrix}$. The matrix $A^5_1$ is not $SL_2(\Z)$-conjugate to $\pm A^5_2$, hence we learn from Proposition \ref{Znsqfree} that $\mathcal{Z}_5(\R)=J(\Cg_{A^5_1})\cup J(\Cg_{A^5_2})$; furthermore each $J(\Cg_{A^5_i})$ is a compact curve (in the Euclidean topology) properly contained in $\mathcal{Z}_5(\R)$. We have
\begin{equation*}
\Cg_{A^5_1}: X^2+Y^2=5, \; \; \Cg_{A^5_2}: X^2+Y^2-X-1=0.
\end{equation*}
In particular $(2, 1) \in \Cg_{A^5_1}$ and $(0, 1) \in \Cg_{A^5_2}$, so the curves $J(\Cg_{A^5_1})$ and $J(\Cg_{A^5_2})$ intersect at $j(2+i)=j(i)$. Finally, it follows from Theorem \ref{modao} below that each $J(\Cg_{A^5_i})$ is not algebraic.

\subsubsection{The curve $\mathcal{Z}_{10}$} Let us finally describe the curve $\mathcal{Z}_{10}$. The narrow class group of $\Z[\sqrt{10}]$ coincides with its class group, and it has two elements. The two matrices
\begin{equation*}
A^{10}_1=\begin{pmatrix}
0 & 10\\
1 & 0
\end{pmatrix}, \; A^{10}_2=\begin{pmatrix}
0 & 5\\
2 & 0
\end{pmatrix}
\end{equation*}
are not $SL_2(\Z)$-conjugate; better, $A^{10}_1$ is not conjugate to $\pm A^{10}_2$. The curve $\Cg_{A^{10}_1}$ (resp. $\Cg_{A^{10}_2}$) is the upper half-circle with equation $X^2+Y^2=10$ (resp. $X^2+Y^2=\frac{5}{2}$). Proposition \ref{Znsqfree} tells us that $J(\Cg_{A^{10}_1})\cup J(\Cg_{A^{10}_2})=\mathcal{Z}_{10}(\R)$, and the inclusions $J(\Cg_{A^{10}_i})\subset \mathcal{Z}_{10}(\R)$ are strict. Notice that $(3, 1) \in \Cg_{A^{10}_1}$ and $\left(\frac{3}{2}, \frac{1}{2}\right) \in \Cg_{A^{10}_2}$. As $\frac{3}{2}+\frac{i}{2}=\frac{2i+1}{i+1}$ we have $j(3+i)=j(i)=j\left(\frac{3}{2}+\frac{i}{2}\right)$, hence the curves $J(\Cg_{A^{10}_i})$ intersect at $1728$. The situation is represented in the picture below, where the curve $J(\Cg_{A^{10}_1})$ (resp. $J(\Cg_{A^{10}_2})$) is depicted in blue (resp. red).

\begin{center}
\begin{figure}[h!]
\includegraphics[width=7.2 cm]{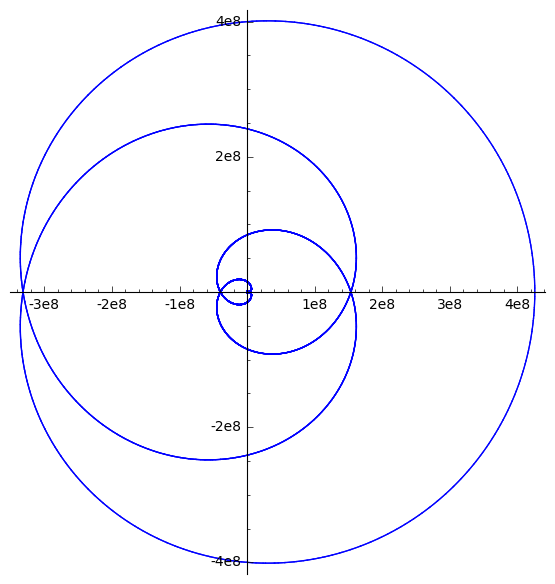} \includegraphics[width=7.2 cm]{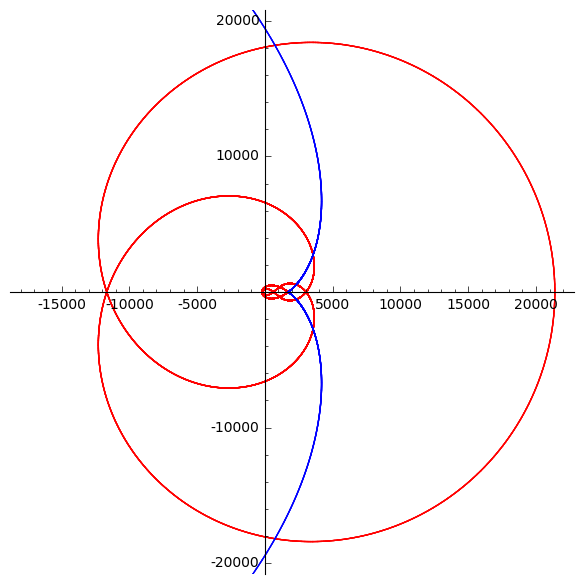}
\end{figure}
\end{center}

As before, the following André-Oort type statement in our situation implies that each $J(\Cg_{A^{10}_i})$ is not algebraic.

\begin{thm}\label{modao}
Let $\mathcal{C}\subset \mathbf{A}^2_\R$ be an irreducible curve. Assume that $\mathcal{C}(\R)$ contains infinitely many images of special points in $\Hp$. Then $\mathcal{C}(\R)=\mathcal{Z}_N(\R)$ for some $N \geq 1$.
\end{thm}

\begin{proof}
We can write $\mathcal{C}(\R)=\{(x, y) \in \R^2 \mid P(x, y)=0\}$ for some geometrically irreducible polynomial $P(X, Y) \in \R[X, Y]$. Let $\mathcal{C}_\C\subset \mathbf{A}^2_\C$ be the base change of $\mathcal{C}$ to $\C$. Consider the map
\begin{align*}
f: \C^2 & \rightarrow \C^2\\
(x, y) & \mapsto (x+iy, x-iy).
\end{align*}
Observe that $(x, y) \in \R^2$ is the image of a special point in $\Hp$ if and only if $f(x, y) \in \C^2$ is the image of a couple of special points in $\Hp \times \Hp$. Therefore the image $\tilde{\mathcal{C}}=f(\mathcal{C}_\C)$ is an irrreducible algebraic curve in $\C^2$ containing infinitely many special points. The curve $\tilde{\mathcal{C}}$ cannot be vertical: indeed, in this case there would be $c \in \C$ such that $x+iy=c$ for every $(x, y) \in \mathcal{C}(\R)$, contradicting the fact that $\mathcal{C}(\R)$ is infinite. For the same reason $\tilde{\mathcal{C}}$ cannot be horizontal. Hence by \cite{and98} the curve $\tilde{C}$ is the vanishing locus of a modular polynomial $\Phi_N(X, Y)$. Therefore
\begin{equation*}
\mathcal{C}(\R)=\{(x, y) \in \R^2 \mid \Phi_N(f(x, y))=0\}=\mathcal{Z}_N(\R).
\end{equation*}
\end{proof}

\bibliographystyle{amsalpha}
\bibliography{bialg}

\Addr

\end{document}